\documentclass[reqno]{amsart}

\usepackage[T1]{fontenc}
\usepackage{amssymb}
\usepackage{enumitem}
\usepackage{mathrsfs}
\usepackage[colorlinks, linkcolor=blue, citecolor=blue, urlcolor=blue]{hyperref}
\usepackage{tikz}

\newtheorem{theorem}{Theorem}[section]

\newtheorem{lemma}[theorem]{Lemma}

\theoremstyle{definition}
\newtheorem{definition}[theorem]{Definition}

\numberwithin{equation}{section}

\DeclareMathOperator{\ran}{ran}

\title{A note on dual Dedekind finiteness}

\author{Ruihuan Mao}
\address{School of Mathematics and Statistics\\
Wuhan University\\
No.~299 Bayi Road\\
Wuhan 430072\\
Hubei Province\\
People's Republic of China}
\email{rhmao2003@outlook.com}

\author{Guozhen Shen}
\address{School of Philosophy\\
Wuhan University\\
No.~299 Bayi Road\\
Wuhan 430072\\
Hubei Province\\
People's Republic of China}
\email{shen\_guozhen@outlook.com}

\thanks{Shen was partially supported by NSFC No.~12101466.}



\begin{document}

\begin{abstract}
A set $A$ is dually Dedekind finite if every surjection from $A$ onto $A$ is injective;
otherwise, $A$ is dually Dedekind infinite. It is proved consistent with $\mathsf{ZF}$
(i.e., the Zermelo--Fraenkel set theory without the axiom of choice)
that there exists a family $\langle A_n\rangle_{n\in\omega}$ of sets such that, for all $n\in\omega$,
$A_n^n$ is dually Dedekind finite whereas $A_n^{n+1}$ is dually Dedekind infinite.
This resolves a question that was left open in [J. Truss, \emph{Fund. Math.} 84, 187--208 (1974)].
\end{abstract}

\maketitle

\section{Introduction}
In 1888, Dedekind~\cite{Dedekind1888} defined an infinite set as a set that is equinumerous with a proper subset of itself.
Such infinite sets are referred to as \emph{Dedekind infinite} today.
Sets that are not Dedekind infinite are called \emph{Dedekind finite} sets.
In the presence of the axiom of choice, a Dedekind finite set is the same thing as a finite set,
where a finite set is defined as one that is equinumerous with a natural number.
However, without the axiom of choice, there may exist a Dedekind-finite infinite set.

Beginning with Tarski~\cite{Tarski1924} and Levy~\cite{Levy1958}, various notions of finiteness have been
explored in the literature, particularly in~\cite{Truss1974,Spivsiak1988,Howard1989,Degen1994,Herrlich2011}.
Among these notions, two are particularly significant: ``power Dedekind finite'' and ``dually Dedekind finite,''
which are defined as follows.

\begin{definition}
Let $A$ be an arbitrary set.
\begin{enumerate}
  \item $A$ is \emph{power Dedekind finite} if $\mathscr{P}(A)$, the powerset of $A$, is Dedekind finite;
        otherwise, $A$ is \emph{power Dedekind infinite}.
  \item $A$ is \emph{dually Dedekind finite} if every surjection from $A$ onto $A$ is injective;
        otherwise, $A$ is \emph{dually Dedekind infinite}.
\end{enumerate}
\end{definition}

Every power Dedekind finite set is dually Dedekind finite, and every dually Dedekind finite set is Dedekind finite.
However, the converses of these implications cannot be proved in $\mathsf{ZF}$ (see~\cite{Truss1974}).
It is also proved in~\cite{Truss1974} that the classes of power Dedekind finite sets and Dedekind finite sets are
closed under unions and disjoint unions, respectively, and therefore are closed under products.
In contrast, although the class of dually Dedekind finite sets is closed under finite unions,
it may not be closed under products without the axiom of choice (see~\cite[Theorem~2(iii)]{Truss1974});
that is, it is consistent with $\mathsf{ZF}$ that there exist dually Dedekind finite sets $A,B$ such that
$A\times B$ is dually Dedekind infinite.

It is natural to ask whether the class of dually Dedekind finite sets is closed under powers.
More precisely, if $A^n$ is dually Dedekind finite, must $A^{n+1}$ also be dually Dedekind finite?
After~\cite{Truss1974}, various properties of dually Dedekind finite sets have been studied
in~\cite{Degen1994,Goldstern1997,Herrlich2015,Panasawatwong2023}, but the above question remains open.
In this article, we provide a negative solution to the above question.

\begin{theorem}\label{msmain}
It is consistent with $\mathsf{ZF}$ that there exists a family $\langle A_n\rangle_{n\in\omega}$ of sets such that,
for all $n\in\omega$, $A_n^n$ is dually Dedekind finite whereas $A_n^{n+1}$ is dually Dedekind infinite.
\end{theorem}

\section{The main construction}
We shall employ the method of permutation models.
We refer the reader to~\cite[Chap.~8]{Halbeisen2017} or~\cite[Chap.~4]{Jech1973}
for an introduction to the theory of permutation models.
Permutation models are not models of $\mathsf{ZF}$;
they are models of $\mathsf{ZFA}$ (the Zermelo--Fraenkel set theory with atoms).
We shall construct a permutation model in which the set of atoms
is a disjoint union $A=\bigcup_{n\in\omega}A_n$, such that for every $n\in\omega$,
$A_n^n$~is dually Dedekind finite, whereas $A_n^{n+1}$ is dually Dedekind infinite.
Then, by the Jech--Sochor transfer theorem (see~\cite[Theorem~17.2]{Halbeisen2017} or~\cite[Theorem~6.1]{Jech1973}),
we arrive at Theorem~\ref{msmain}.

We construct the set $A$ of atoms as follows:
\[
A=\bigcup_{n\in\omega}A_n,
\]
where
\[
A_n=\{a_{n,s}\mid s\text{ is a nonempty finite string of numbers }0,\dots,n\}.
\]
For simplicity, we write $[s]_n$ for $a_{n,s}$. For each $n\in\omega$ and each $i\leqslant n$, let
\[
A_{n,i}=\{[i^\smallfrown s]_n\mid s\text{ is a finite string of numbers }0,\dots,n\}.
\]
In order to construct the group $\mathcal{G}$ of permutations of $A$ needed for the permutation model, we define,
for each $n\in\omega$, a function $f_n:\prod_{i\leqslant n}A_{n,i}\to\{\varnothing\}\cup\prod_{i\leqslant n}A_{n,i}$ as follows:
\begin{align*}
  f_n([0j_0]_n,\dots,[nj_n]_n) & =\varnothing, \\
  f_n([0^\smallfrown s_0{}^\smallfrown k_0l_0]_n,\dots,[n^\smallfrown s_n{}^\smallfrown k_nl_n]_n) &
  =([0^\smallfrown s_0{}^\smallfrown m_0]_n,\dots,[n^\smallfrown s_n{}^\smallfrown m_n]_n),
\intertext{where $s_0,\dots,s_n$ are finite strings of the same length, and for every $i\leqslant n$, $m_i\equiv\sum_{j\neq i}k_j\pmod{n+1}$,}
  f_n(x) & =x,\quad\text{otherwise.}
\end{align*}
It is easy to verify that $f_n$ is a surjection for all $n\in\omega$.

Every permutation $\pi$ of $A$ extents recursively to a permutation of the universe by
\[
\pi x=\{\pi z\mid z\in x\}.
\]
Now, let $\mathcal{G}$ be the group of all permutations $\pi$ of $A$ such that for every $n\in\omega$,
\begin{enumerate}
  \item $\pi$ fixes $A_{n,i}$ setwise for all $i\leqslant n$,
  \item for all $[s]_n\in A_n$, if $\pi[s]_n=[t]_n$, then $s$ and $t$ have the same length,
  \item $\pi f_n=f_n$.
\end{enumerate}
Then $x$ belongs to the permutation model $\mathcal{V}$ determined by $\mathcal{G}$ if and only if
$x\subseteq\mathcal{V}$ and $x$ has a \emph{finite support}, that is, a finite subset $E\subseteq A$
such that every permutation $\pi\in\mathcal{G}$ fixing $E$ pointwise also fixes $x$.
Note that, for every $n\in\omega$ and every $i\leqslant n$, $A_{n,i}$ and $f_n$ are
fixed by every permutation in $\mathcal{G}$, so $A_{n,i},f_n\in\mathcal{V}$.

\begin{lemma}\label{sh21}
In $\mathcal{V}$, for every $n\in\omega$, $\prod_{i\leqslant n}A_{n,i}$ is dually Dedekind infinite.
\end{lemma}
\begin{proof}
The surjection $f_n$ witnesses the dual Dedekind infiniteness of $\prod_{i\leqslant n}A_{n,i}$.
\end{proof}

The key step in the proof is the following lemma.

\begin{lemma}\label{sh22}
In $\mathcal{V}$, for every $n\in\omega$ and every $e_0,\dots,e_n\in\omega$,
if at least one of $e_0,\dots,e_n$ is zero, then $\prod_{i\leqslant n}A_{n,i}^{e_i}$ is dually Dedekind finite.
\end{lemma}

To prove this lemma, we require the following two auxiliary lemmas.

\begin{lemma}\label{sh23}
Let $X=\bigcup_{k\in\omega}X_k$, where $X_k$ is a power Dedekind finite set for all $k\in\omega$.
If $X$ is dually Dedekind infinite and $o\notin X$, then there exist a surjection $g:X\to\{o\}\cup X$
and an increasing function $r:\omega\to\omega$ such that for every $k\in\omega$,
\[
g^{-k-1}[\{o\}]\subseteq X_{r(k)}.
\]
\end{lemma}
\begin{proof}
See~\cite[Lemma~8]{Truss1974}.
\end{proof}

\begin{lemma}\label{sh25}
Let $n\in\omega$. For each $i<n$, let $e_i\in\omega$ and let $X_i=\bigcup_{l\in\omega}X_{i,l}$, where $X_{i,l}$ is
power Dedekind finite for all $l\in\omega$. If $X=\prod_{i<n}X_i^{e_i}$ is dually Dedekind infinite and $o\notin X$,
then there exist a surjection $g:X\to\{o\}\cup X$ and functions $h,p_{i,j}:\omega\to\omega$ \textup{($i<n,j<e_i$)} such that
\begin{enumerate}[label=\upshape(\arabic*)]
  \item $h$ is increasing,
  \item every $p_{i,j}$ is monotonic, and at least one of $p_{i,j}$ is unbounded,
  \item for every $k\in\omega$, $g^{-h(k)-1}[\{o\}]\subseteq\prod_{i<n}\prod_{j<e_i}X_{i,p_{i,j}(k)}$.
\end{enumerate}
\end{lemma}
\begin{proof}
Let $I=\{(i,j)\mid i<n,j<e_i\}$, let $Q$ be the set of all functions from $I$ to $\omega$,
and let $f$ be a bijection between $\omega$ and $Q$. Note that
\[
X=\prod_{i<n}X_i^{e_i}=\bigcup_{q\in Q}\prod_{i<n}\prod_{j<e_i}X_{i,q(i,j)}=\bigcup_{k\in\omega}\prod_{i<n}\prod_{j<e_i}X_{i,f(k)(i,j)}.
\]
Since power Dedekind finite sets are closed under finite products, it follows that
$\prod_{i<n}\prod_{j<e_i}X_{i,f(k)(i,j)}$ is power Dedekind finite for all $k\in\omega$.
Thus, by Lemma~\ref{sh23}, there exist a surjection $g:X\to\{o\}\cup X$ and
an increasing function $r:\omega\to\omega$ such that for every $k\in\omega$,
\begin{equation}\label{sh11}
g^{-k-1}[\{o\}]\subseteq\prod_{i<n}\prod_{j<e_i}X_{i,f(r(k))(i,j)}.
\end{equation}

Let $R$ be the componentwise ordering of $Q$; that is, for all $p,q\in Q$,
$pRq$ if and only if $p(i,j)\leqslant q(i,j)$ for all $(i,j)\in I$.
It is well known that $R$ is a well-quasi-ordering of $Q$,
and hence there exists an increasing function $h:\omega\to\omega$
such that $f(r(h(k)))\mathrel{R}f(r(h(k+1)))$ for all $k\in\omega$.
For each $(i,j)\in I$, let $p_{i,j}:\omega\to\omega$ be defined by $p_{i,j}(k)=f(r(h(k)))(i,j)$.
Then every $p_{i,j}$ is monotonic. Since $r$ is increasing, at least one of $p_{i,j}$ is unbounded.
Finally, it follows from \eqref{sh11} that $g^{-h(k)-1}[\{o\}]\subseteq\prod_{i<n}\prod_{j<e_i}X_{i,p_{i,j}(k)}$ for all $k\in\omega$.
\end{proof}

\begin{proof}[Proof of Lemma~\ref{sh22}]
By symmetry, without loss of generality, assume that $e_n=0$. For each $l\in\omega$, let
\[
A_{n,i,l}=\{[i^\smallfrown s]_n\in A_{n,i}\mid s\text{ has length }l\}.
\]
Then each $A_{n,i,l}$ is a finite set in $\mathcal{V}$ and $A_{n,i}=\bigcup_{l\in\omega}A_{n,i,l}$.
Assume toward a contradiction that $\prod_{i<n}A_{n,i}^{e_i}$ is dually Dedekind infinite.
By Lemma~\ref{sh25}, in $\mathcal{V}$, there exist a surjection $g:\prod_{i<n}A_{n,i}^{e_i}\to\{\varnothing\}\cup\prod_{i<n}A_{n,i}^{e_i}$
and functions $h,p_{i,j}:\omega\to\omega$ ($i<n,j<e_i$) such that
\begin{enumerate}
  \item $h$ is increasing,
  \item every $p_{i,j}$ is monotonic, and at least one of $p_{i,j}$ is unbounded,
  \item for every $k\in\omega$,
        \begin{equation}\label{sh10}
          g^{-h(k)-1}[\{\varnothing\}]\subseteq\prod_{i<n}\prod_{j<e_i}A_{n,i,p_{i,j}(k)}.
        \end{equation}
\end{enumerate}
Without loss of generality, assume that $p_{0,0}$ is unbounded. Let
\begin{align*}
  B & =\{(i,j)\mid p_{i,j}\text{ is bounded}\}, \\
  C & =\{(i,j)\mid p_{i,j}\text{ is unbounded}\}.
\end{align*}
Then $(0,0)\in C$. For each $(i,j)\in B$, let $b_{i,j}$ be the least upper bound of $\ran(p_{i,j})$.

Since $g\in\mathcal{V}$, there exists a finite support $E\subseteq A$ for $g$.
Since $p_{0,0}$ is unbounded, there is a least $q\in\omega$ such that $p_{0,0}(q)>0$,
$p_{0,0}(q)>b_{i,j}$ for all $(i,j)\in B$, and $E\cap A_n\subseteq\bigcup_{i\leqslant n}\bigcup_{l<p_{0,0}(q)}A_{n,i,l}$.
We define a permutation $\sigma$ of $A$ as follows:
\begin{align*}
  \sigma[i^\smallfrown s^\smallfrown k]_n               & =[i^\smallfrown s^\smallfrown k']_n, \\
  \sigma[n^\smallfrown s^\smallfrown k^\smallfrown t]_n & =[n^\smallfrown s^\smallfrown k'{}^\smallfrown t]_n,
\intertext{where $i<n$, $s$ is a finite string of length $p_{0,0}(q)-1$, $t$ is a nonempty finite string, and $k'\equiv k+1\pmod{n+1}$,}
  \sigma a                                              & =a,\quad\text{otherwise.}
\end{align*}

Clearly, $\sigma$ satisfies conditions (1) and (2) in the definition of $\mathcal{G}$.
Since $\sigma$ is the identity on $A\setminus A_n$, once we prove $\sigma f_n=f_n$,
it follows that $\sigma\in\mathcal{G}$, which allows us to derive a contradiction as follows.

Let $r$ be the least natural number such that $r>q$ and $p_{i,j}(r)>p_{0,0}(q)$ for all $(i,j)\in C$.
Since $g$ is surjective, there are $u,v\in\prod_{i<n}A_{n,i}^{e_i}$ such that $g^{h(q)+1}(v)=\varnothing$
and $g^{h(r)-h(q)}(u)=v$. Then $g^{h(r)+1}(u)=\varnothing$. Hence, by~\eqref{sh10},
\begin{align*}
  u & \in\prod_{i<n}\prod_{j<e_i}A_{n,i,p_{i,j}(r)}, \\
  v & \in\prod_{i<n}\prod_{j<e_i}A_{n,i,p_{i,j}(q)}.
\end{align*}
For all $(i,j)\in B$, $p_{i,j}(r)\leqslant b_{i,j}<p_{0,0}(q)$, and for all $(i,j)\in C$, $p_{i,j}(r)>p_{0,0}(q)$.
Therefore, by the definition of $\sigma$, $\sigma u=u$. Since $E\cap A_n\subseteq\bigcup_{i\leqslant n}\bigcup_{l<p_{0,0}(q)}A_{n,i,l}$,
it follows that $\sigma$ fixes $E$ pointwise, and thus $\sigma g=g$, which implies that $\sigma v=v$ since $v=g^{h(r)-h(q)}(u)$.
But $v\in\prod_{i<n}\prod_{j<e_i}A_{n,i,p_{i,j}(q)}$, and by the definition of $\sigma$, $\sigma v\neq v$, which is a contradiction.

It still remains to show that $\sigma f_n=f_n$. Take an arbitrary $x\in\prod_{i\leqslant n}A_{n,i}$.
We have to prove that $\sigma(f_n(x))=f_n(\sigma x)$. Consider the following three cases.

\textsc{Case~1}. $x$ is of the form $([0j_0]_n,\dots,[nj_n]_n)$.
Then $\sigma x$ is also of this form. Hence $\sigma(f_n(x))=\varnothing=f_n(\sigma x)$.

\textsc{Case~2}. $x$ is of the form $([0^\smallfrown s_0{}^\smallfrown k_0l_0]_n,\dots,[n^\smallfrown s_n{}^\smallfrown k_nl_n]_n)$,
where $s_0,\dots,s_n$ are finite strings of the same length, say $l$.
Let $f_n(x)=([0^\smallfrown s_0{}^\smallfrown m_0]_n,\dots,[n^\smallfrown s_n{}^\smallfrown m_n]_n)$,
where $m_i\equiv\sum_{j\neq i}k_j\pmod{n+1}$ for all $i\leqslant n$. Consider further the following four subcases.

\textsc{Case~2a}. $l<p_{0,0}(q)-2$. Then $\sigma x=x$, and hence $\sigma(f_n(x))=f_n(x)=f_n(\sigma x)$.

\textsc{Case~2b}. $l=p_{0,0}(q)-2$.
Then $\sigma x=([0^\smallfrown s_0{}^\smallfrown k_0l_0']_n,\dots,[n^\smallfrown s_n{}^\smallfrown k_nl_n']_n)$,
where $l_i'\equiv l_i+1\pmod{n+1}$ for all $i<n$ and $l_n'=l_n$. Thus $f_n(\sigma x)=f_n(x)=\sigma(f_n(x))$.

\textsc{Case~2c}. $l=p_{0,0}(q)-1$. Then
\begin{align*}
  \sigma x       & =([0^\smallfrown s_0{}^\smallfrown k_0'l_0]_n,\dots,[n^\smallfrown s_n{}^\smallfrown k_n'l_n]_n), \\
  \sigma(f_n(x)) & =([0^\smallfrown s_0{}^\smallfrown m_0']_n,\dots,[n^\smallfrown s_n{}^\smallfrown m_n']_n),
\end{align*}
where $k_i'=k_i$ and $m_i'\equiv m_i+1\pmod{n+1}$ for all $i<n$, $k_n'\equiv k_n+1\pmod{n+1}$, and $m_n'=m_n$.
It is easy to verify that $m_i'\equiv\sum_{j\neq i}k_j'\pmod{n+1}$ for all $i\leqslant n$. Hence, $f_n(\sigma x)=\sigma(f_n(x))$.

\textsc{Case~2d}. $l\geqslant p_{0,0}(q)$.
Then
\begin{align*}
  \sigma x       & =([0^\smallfrown s_0'{}^\smallfrown k_0l_0]_n,\dots,[n^\smallfrown s_n'{}^\smallfrown k_nl_n]_n), \\
  \sigma(f_n(x)) & =([0^\smallfrown s_0'{}^\smallfrown m_0]_n,\dots,[n^\smallfrown s_n'{}^\smallfrown m_n]_n),
\end{align*}
where $s_i'=s_i$ for all $i<n$ and $s_n'$ is obtained from $s_n$ by replacing its $p_{0,0}(q)$-th term $j$
with the remainder of $j+1$ module $n+1$. Hence, $f_n(\sigma x)=\sigma(f_n(x))$.

\textsc{Case~3}. Otherwise, it is easy to see that $\sigma(f_n(x))=\sigma x=f_n(\sigma x)$.
\end{proof}

Now, we are ready to prove the main theorem.

\begin{theorem}\label{sh26}
In $\mathcal{V}$, for every $n\in\omega$, $A_n^n$ is dually Dedekind finite, and $A_n^{n+1}$ is dually Dedekind infinite.
\end{theorem}
\begin{proof}
Note that $A_n=\bigcup_{i\leqslant n}A_{n,i}$. Hence, by Lemma~\ref{sh21},
$A_n^{n+1}\supseteq\prod_{i\leqslant n}A_{n,i}$ is dually Dedekind infinite. Note also that
\[
A_n^n=\bigcup_{j_0,\dots,j_{n-1}\leqslant n}\prod_{k<n}A_{n,j_k},
\]
and for every $j_0,\dots,j_{n-1}\leqslant n$,
\[
\prod_{k<n}A_{n,j_k}\approx\prod_{i\leqslant n}A_{n,i}^{e_i},
\]
where $e_i=|\{k<n\mid i=j_k\}|$ for all $i\leqslant n$.
Since $e_0+\dots+e_n=n$, at least one of $e_0,\dots,e_n$ is zero.
Hence, by Lemma~\ref{sh22}, $\prod_{i\leqslant n}A_{n,i}^{e_i}$ is dually Dedekind finite.
Since dually Dedekind finite sets are closed under finite unions, it follows that $A_n^n$ is dually Dedekind finite.
\end{proof}

Finally, Theorem~\ref{msmain} follows immediately from Theorem~\ref{sh26} and the Jech--Sochor transfer theorem.

\end{document}